\documentclass{amsart}
\usepackage{graphicx}
\usepackage{amsmath}
\usepackage{amsthm}
\usepackage{amsfonts}
\usepackage{amssymb, amscd}

\usepackage{color}

\usepackage[all]{xy}

\newtheorem{thm}{Theorem}[section]
\newtheorem{cor}[thm]{Corollary}
\newtheorem{lem}[thm]{Lemma}
\newtheorem{proposition}[thm]{Proposition}
\newtheorem{example}[thm]{Example}
\theoremstyle{definition}
\newtheorem{definition}[thm]{Definition}
\theoremstyle{remark}
\newtheorem{remark}[thm]{Remark}
\numberwithin{equation}{section}

\newcommand{\dl}{\displaystyle}

\begin{document}

\title[Some constructions of multiplicative $n$-ary Hom-Nambu algebras]
{ Some constructions of multiplicative  $n$-ary Hom-Nambu algebras}%
\author{Abdelkader Ben Hassine,  Sami Mabrouk, Othmen Ncib }%
\address{Abdelkader Ben Hassine,Faculty of Sciences, University of Sfax,   BP
1171, 3000 Sfax, Tunisia\ \ \ \ \ \ \ \ \ \ \ \ \ \ \ \ \ \ \ \ \ \ \ \&
Department of Mathematics, Faculty of Science and Arts at
Belqarn, P. O. Box 60, Sabt Al-Alaya 61985, Bisha University, Saudi Arabia}%
\email{benhassine.abdelkader@yahoo.fr}
 \address{Sami Mabrouk, Othmen Ncib,  Faculty of Sciences, University of Gafsa,   BP
2100, Gafsa, Tunisia}%
\email{Mabrouksami00@yahoo.fr, othmenncib@yahoo.fr}


\date{}
%
\maketitle
\begin{abstract}

We show that  given a  Hom-Lie algebra one can construct the n-ary Hom-Lie
bracket by means of an (n-2)-cochain of given Hom-Lie algebra and find the
conditions under which this n-ary bracket satisfies the Filippov-Jacobi
identity, there by inducing the structure of n-Hom-Lie algebra. We introduce the notion of  Hom-Lie $n$-uplet system which is the generalization of Hom-Lie triple system. We construct  Hom-Lie $n$-uplet system using a Hom-Lie algebra.
\end{abstract}


\section*{Introduction}

\

\hspace{0.5cm}The first instances of n-ary algebras in Physics appeared with a generalization of the Hamiltonian mechanics proposed
in 1973 by Nambu \cite{Nam}. More recent motivation comes from string theory and M-branes involving naturally an algebra with
ternary operation called Bagger-Lambert algebra which gives impulse to a significant development. It was used in \cite{Bagger&Lambert} as
one of the main ingredients in the construction of a new type of supersymmetric gauge theory that is consistent with all the
symmetries expected of a multiple M2-brane theory: 16 supersymmetries, conformal invariance, and an SO(8) R-symmetry
that acts on the eight transverse scalars. On the other hand in the study of supergravity solutions describing M2-branes
ending on M5-branes, the Lie algebra appearing in the original Nahm equations has to be replaced with a generalization
involving ternary bracket in the lifted Nahm equations(see \cite{Basu&Harvey}).


\hspace{0.5cm} In \cite{Ataguema&Makhlouf&Silvestrov}, generalizations of n-ary algebras
of Lie type and associative type by twisting the identities using linear
maps have been introduced. The notion  of representations, derivations, cohomology and deformations are studies in \cite{Arnlind&kitoni&makhlouf&Silvi,Daletskii&Takhtajan,Gautheron,Ma&Chen, Takhtajan1}. These generalizations include n-ary Hom-algebra
structures generalizing the n-ary algebras of Lie type including n-ary Nambu
algebras, n-Lie algebra (called also n-ary Nambu-Lie algebras)  and n-ary Lie algebras, and n-ary algebras
of associative type including n-ary totally associative and n-ary partially
associative algebras. In \cite{Arnlind&Makhlouf&Silvestrov}, a method was demonstrated of how to construct
ternary multiplications from the binary multiplication of a Hom-Lie algebra,
a linear twisting map, and a trace function satisfying certain compatibility
conditions; and it was shown that this method can be used to construct
ternary Hom-Nambu-Lie algebras from Hom-Lie algebras. This construction
was generalized to n-Lie algebras and n-Hom-Nambu-Lie algebras in \cite{Arnlind&Makhlouf&Silvestrov1}.

\

\hspace{0.5cm} It is well known that the algebras of derivations and generalized derivations are very important in the
study of Lie algebras and its generalizations. The notion of $\delta$-derivations appeared in the paper of
Filippov \cite{Filippov}. The results for $\delta$-derivations and generalized derivations were studied by many authors.
For example, R.Zhang and Y.Zhang \cite{Zhang&Zhang} generalized the above results to the case of Lie superalgebras; Chen,
Ma, Ni and Zhou considered the generalized derivations of color Lie algebras, Hom-Lie superalgebras
and Lie triple systems \cite{Chen&Ma&Ni,Chen&Ma&Zhou}. Derivations and generalized derivations of n-ary algebras were considered
in \cite{Kaygorodov,Kaygorodov1} and other. In \cite{Beites&Kaygorodov&Popov}, the authors generalize this results in  Color n-ary Nambu Hom case.

\

\hspace{0.5cm}This paper is organized as follows. In Section I we review basic concepts of Hom-Lie, $n$-ary Hom-Nambu algebras and $n$-Hom-Lie algebras. We also recall the definitions of derivations, $\alpha^k$-derivations, $\alpha^k$-quasiderivations and $\alpha^k$-centroid. In Section II we provide a construction
procedure of n-Hom-Lie algebras starting from a binary
bracket of a Hom-Lie algebra and multilinear  form satisfying certain  conditions. To this end, we give the relation between $\alpha^k$-derivations, (resp. $\alpha^k$-quasiderivations and $\alpha^k$-centroid) of   Hom-Lie algebra and  $\alpha^k$-derivations, (resp. $\alpha^k$-quasiderivations and $\alpha^k$-centroid) of n-Hom-Lie algebras. In Section III we introduce the notion of Hom-Lie $n$-uplet system which is the generalisation of Lie $n$-uplet system which is introduced in \cite{Azcarraga&Izquierdo}. We construct  Hom-Lie $n$-uplet system using Hom-Lie algebra. Finally, we give a relation between $\alpha^k$-quasiderivations of a Hom-Lie algebra and $(n + 1)$-ary
$\alpha^k$-derivation of  Hom-Lie $n$-uplet system associate.

\section{Hom-Lie algebras and n-ary Hom-Nambu algebras}

\

\hspace{0.5cm}  Throughout this paper, we will for simplicity of exposition assume that $\mathbb{K}$ is an algebraically closed
field of characteristic zero, even though for most of the general definitions and results in the paper this
assumption is not essential.
\subsection{Definitions}
The notion of a Hom-Lie algebra was initially motivated by examples of deformed Lie
algebras coming from twisted discretizations of vector fields(see \cite{Hartwig&Larsson&Silvestrov,Larsson&Silvestrov}). We will follow notation conventions
in \cite{Makhlouf&Silvestrov1}.
\begin{definition}
A Hom-Lie algebra is a triple $(\mathfrak g, [~,~],\alpha)$, where $[~
,~]:\mathfrak g\times \mathfrak g\rightarrow\mathfrak g$ is a bilinear
map and $\alpha:\mathfrak g\rightarrow\mathfrak g$ a linear map satisfying
$$
[x,y] = - [y,x] ~~~~~~~~~~~~~~~~~~~~\text{(skew-symmetry)}$$
$$\displaystyle\circlearrowleft_{x,y,z}[\alpha(x),[y,z]]=0~~~~~~~~~~~~~~~~~~~~~~~\text{(Hom-Jacobi~ condition)}$$
for all x, y ,z from $\mathfrak g$, where $\displaystyle\circlearrowleft_{x,y,z}$ denotes summation over the cyclic permutations of $x, y ,z$.

\end{definition}
\begin{definition}
  A Hom-Lie algebras $(\mathfrak{g},[~,~],\alpha)$ is called multiplicative if $\alpha([x,y])=[\alpha(x),\alpha(y)]$ for all $x,y\in \mathfrak{g}$.
\end{definition}
We define a linear map $ad:\mathfrak{g}\rightarrow End(\mathfrak{g})$ by $\text{ad}_x(y)=[x,y]$. Thus the Hom-Jacobi identity is equivalent to  \begin{equation}\label{RepresentationAdjointe}
\text{ad}_{[x,y]}(\alpha(z))=\text{ad}_{\alpha(x)}\circ \text{ad}_y(z)- \text{ad}_{\alpha(y)}\circ \text{ad}_x(z),\ \textrm{for\ all} \ x,y,z\in\mathfrak{g}.
           \end{equation}
           \begin{remark}
             An ordinary Lie algebras is a Hom-Lie algebas when $\alpha=id$.
           \end{remark}
           \begin{example}
             Let $\mathcal{A}$ be the complex algebra  where $\mathcal{A}= \mathbb{C}[t, t^{-1}]$ is the Laurent polynomials
in one variable.  The generators of $\mathcal{A}$ are of the form $t^n$ for $n \in \mathbb{Z}$.\\
Let $q\in \mathbb{C}\backslash\{0, 1\}$ and $n\in  \mathbb{N}$, we set $\{n\} = \frac{1-q^n}{1-q}$, a $q$-number. The $q$-numbers have the following
properties $\{n + 1\} = 1 + q\{n\} = \{n\} + q^n$ and $\{n + m\} = \{n\} + q^n\{m\}$.

Let $\mathfrak{A}_q$  be a space with basis $\{L_m,\ I_m| m\in\mathbb{Z}\}$  where $L_m=-t^mD,\ I_m=-t^m$ and $D$ is a $q$-derivation on $\mathcal{A}$ such that
$$D(t^m)=\{m\}t^m.$$
We define the bracket  $[\ ,\ ]_q:\mathfrak{A}_q\times\mathfrak{A}_q\longrightarrow\mathfrak{A}_q$, with respect the super-skew-symmetry for $n,m\in\mathbb{Z}$  by
\begin{align}
&\label{crochet1}[L_m,L_n]_q=(\{m\}-\{n\})L_{m+n},\\
&\label{crochet2}[L_m,I_n]_q=-\{n\}I_{m+n},\\
&\label{crochet2}[I_m,I_n]_q=0.
\end{align}
Let $\alpha$ be an even linear map on $\mathfrak{A}_q$  defined on the generators by
\begin{eqnarray*}
\alpha_q(L_n)&=&(1+q^n)L_n,\hskip0.5cm \alpha_q(I_n)=(1+q^n)I_n,
\end{eqnarray*}
The triple $(\mathfrak{A}_q, [\ ,\ ]_q, \alpha_q)$ is a Hom-Lie algebra, called the
{$q$-deformed Heisenberg-Virasoro algebra of Hom-type}.

           \end{example}
           \begin{example}\label{Sl2HomLieExemple}
             We consider the matrix construction of the algebra $\mathfrak{sl}_2(\mathbb{R})$ generated by the following three vectors
$$H=\left(
      \begin{array}{cc}
        1 & 0 \\
       0 & -1 \\
      \end{array}
    \right)\;;\;
    X=\left(
        \begin{array}{cc}
          0  & 1 \\
          0 & 0 \\
        \end{array}
      \right)\;;\;
      Y=\left(
          \begin{array}{cc}
            0 & 0  \\
            1 & 0  \\
          \end{array}
        \right)
$$
The defining relations are
$$[H,X]=2X\;;\;[H,Y]=-2Y\;;\;[X,Y]=H$$
Let $\lambda\in\mathbb{R}^*$, consider the linear maps $\alpha_{\lambda}:\mathfrak{sl}_2(\mathbb{R})\rightarrow\mathfrak{sl}_2(\mathbb{R})$ defined by
$$\alpha_{\lambda}(H)=H\;\;;\;\;\alpha_{\lambda}(X)=\lambda^2X\;\;;\;\;\alpha_{\lambda}(Y)=\frac{1}{\lambda^2}Y$$
Note that $\alpha_{\lambda}$ is a Lie algebra automorphism.

In \cite{AmmarMakhloufJA2010}, the authors have shown that $(\mathfrak{sl}_2(\mathbb{R}))_\lambda=(\mathfrak{sl}_2(\mathbb{R}),[\ ,\ ]_{\alpha_\lambda},\alpha_\lambda)$ is a family of multiplicative Hom-Lie algebras where the Hom-Lie bracket $[\ ,\ ]_{\alpha_\lambda}$ on the basis elements is given, for $\lambda\neq0$ by
$$[H,X]_{\alpha_\lambda}=2\lambda^2X\;\;;\;\;[H,Y]_{\alpha_\lambda}=-\frac{2}{\lambda^2}Y\;\;;\;\;[X,Y]_{\alpha_\lambda}=H$$

           \end{example}

\

\hspace{0.5cm}Now, we recall the definitions  of $n$-ary Hom-Nambu algebras and $n$-ary Hom-Nambu-Lie algebras, generalizing of $n$-ary Nambu algebras and $n$-ary Nambu-Lie algebras
(called also Filippov algebras) respectively, which were  introduced in \cite{Ataguema&Makhlouf&Silvestrov} by Ataguema, Makhlouf and Silvestrov.
\begin{definition}
An \emph{$n$-ary Hom-Nambu} algebra is a triple $(\mathcal{N}, [\ ,..., \ ],  \widetilde{\alpha} )$ consisting of a vector space  $\mathcal{N}$, an
$n$-linear map $[\ ,..., \ ] :  \mathcal{N}^{ n}\longrightarrow \mathcal{N}$ and a family
$\widetilde{\alpha}=(\alpha_i)_{1\leq i\leq n-1}$ of  linear maps $ \alpha_i:\ \ \mathcal{N}\longrightarrow \mathcal{N}$, satisfying \\
  \begin{eqnarray}\label{NambuIdentity}
  && \big[\alpha_1(x_1),....,\alpha_{n-1}(x_{n-1}),[y_1,....,y_{n}]\big]= \\ \nonumber
&& \sum_{i=1}^{n}\big[\alpha_1(y_1),....,\alpha_{i-1}(y_{i-1}),[x_1,....,x_{n-1},y_i]
  ,\alpha_i(y_{i+1}),...,\alpha_{n-1}(y_n)\big],
  \end{eqnarray}
  for all $(x_1,..., x_{n-1})\in \mathcal{N}^{ n-1}$, $(y_1,...,  y_n)\in \mathcal{N}^{ n}.$\\
  The identity \eqref{NambuIdentity} is called \emph{Hom-Nambu identity}.
  \end{definition}

Let
$X=(x_1,\ldots,x_{n-1})\in \mathcal{N}^{n-1}$, $\widetilde{\alpha}
(X)=(\alpha_1(x_1),\ldots,\alpha_{n-1}(x_{n-1}))\in \mathcal{N}^{n-1}$ and
$y\in \mathcal{N}$. We define an adjoint map  $\text{ad}(X)$ as  a linear map on $\mathcal{N}$,
such that
\begin{equation}\label{adjointMapNaire}
\text{ad}_X(y)=[x_{1},\cdots,x_{n-1},y].
\end{equation}

Then the Hom-Nambu identity \eqref{NambuIdentity} may be written in terms of adjoint map as
\begin{equation*}
\text{ad}_{\widetilde{\alpha} (X)}( [x_{n},...,x_{2n-1}])=
\sum_{i=n}^{2n-1}{[\alpha_1(x_{n}),...,\alpha_{i-n}(x_{i-1}),
\text{ad}_X(x_{i}), \alpha_{i-n+1}(x_{i+1}) ...,\alpha_{n-1}(x_{2n-1})].}
\end{equation*}
\begin{definition}
  An {$n$-ary Hom-Nambu} algebra is a triple $(\mathcal{N}, [\  ,..., \ ],  \widetilde{\alpha} )$ is called $n$-Hom-Lie algebra if the bracket $[\  ,..., \ ]$ is skewsymmetric i.e $[x_{\sigma(1)},\cdots,x_{\sigma(n)}]=(-1)^{sign(\sigma)}[x_{1},\cdots,x_{n}]$ for $\sigma\in S_n$.
\end{definition}
\begin{remark}
When the maps $(\alpha_i)_{1\leq i\leq n-1}$ are all identity maps, one recovers the classical $n$-ary Nambu algebras. The Hom-Nambu identity \eqref{NambuIdentity}, for $n=2$,  corresponds to Hom-Jacobi identity (see \cite{Makhlouf&Silvestrov1}), which reduces to Jacobi identity when $\alpha_1=id$.
\end{remark}

Let $(\mathcal{N},[\ ,\dots,\ ],\widetilde{\alpha})$ and
$(\mathcal{N}',[\cdot,\dots,\cdot]',\widetilde{\alpha}')$ be two $n$-ary Hom-Nambu
algebras  where $\widetilde{\alpha}=(\alpha_{i})_{i=1,\cdots,n-1}$ and
$\widetilde{\alpha}'=(\alpha'_{i})_{i=1,\cdots,n-1}$. A linear map $f:
\mathcal{N}\rightarrow \mathcal{N}'$ is an  $n$-ary Hom-Nambu algebras \emph{morphism}  if it satisfies
\begin{eqnarray*}f ([x_{1},\cdots,x_{2n-1}])&=&
[f (x_{1}),\cdots,f (x_{2n-1})]'\\
f \circ \alpha_i&=&\alpha'_i\circ f \quad \forall i=1,n-1.
\end{eqnarray*}

\hspace{0.5cm}In the sequel we deal sometimes with a particular class of $n$-ary Hom-Nambu algebras which we call $n$-ary multiplicative Hom-Nambu  algebras.

\begin{definition}
A \emph{multiplicative $n$-ary Hom-Nambu algebra }
(resp. \emph{ multiplicative $n$-Hom-Lie algebra}) is an $n$-ary Hom-Nambu algebra  (resp. $n$-Hom-Lie algebra) $(\mathcal{N}, [\  ,..., \ ],  \widetilde{ \alpha})$ with  $\widetilde{\alpha}=(\alpha_i)_{1\leq i\leq n-1}$
where  $\alpha_1=...=\alpha_{n-1}=\alpha$  and satisfying
\begin{equation}
\alpha([x_1,..,x_n])=[\alpha(x_1),..,\alpha(x_n)],\ \  \forall x_1,...,x_n\in \mathcal{N}.
\end{equation}
For simplicity, we will denote the $n$-ary multiplicative Hom-Nambu algebra as $(\mathcal{N}, [\  ,..., \  ],  \alpha)$ where $\alpha :\mathcal{N}\rightarrow \mathcal{N}$ is a linear map. Also by misuse of language an element  $x\in \mathcal{N}^n$ refers to  $x=(x_1,..,x_{n})$ and  $\alpha(x)$ denotes $(\alpha (x_1),\dots,\alpha (x_n))$.
\end{definition}
\subsection{Derivations, Quasiderivations and centroids of multiplicative $n$-Hom-Lie algebras}

\

In this section we recall the definition of derivation, generalized derivation, quasiderivation and centroids of multiplicative $n$-Hom-Lie algebras.\\
Let $(\mathcal{N}, [\   ,\cdots, \  ],  \alpha )$ be an  multiplicative $n$-Hom-Lie algebra. We
denote by $\alpha^k$ the $k$-times composition of $\alpha$
(i.e.  $\alpha^k=\alpha\circ...\circ\alpha$ $k$-times).
In particular $\alpha^{-1}=0$, $\alpha^0=id$.

\begin{definition}
For any $k\geq1$, we call $D\in End(\mathcal{N})$ an $\alpha^k$-\emph{derivation }of the
 multiplicative $n$-Hom-Lie algebra $(\mathcal{N}, [\  ,...,\ ],  \alpha )$ if
\begin{equation}\label{alphaKderiv1}[D,\alpha]=0\ \ \textrm{i.e.}\ \ D\circ\alpha=\alpha\circ D,\end{equation}
and
\begin{equation}\label{alphaKderiv2}
D[x_1,...,x_n]=\sum_{i=1}^n[\alpha^k(x_1),...,\alpha^k(x_{i-1}),D(x_i),\alpha^k(x_{i+1}),...,\alpha^k(x_n)],
\end{equation}
We denote by $Der_{\alpha^k}(\mathcal{N})$ the set of $\alpha^k$-derivations of
the  multiplicative $n$-Hom-Lie  algebra $\mathcal{N}$.
\end{definition}

For $X=(x_1,...,x_{n-1})\in \mathcal{N}^{ n-1}$ satisfying $\alpha(X)=X$ and $k\geq 1$,
we define the map $\text{ad}^k_X\in End(\mathcal{N})$ by
\begin{equation}\label{ad_k(u)}
\text{ad}^k_X(y)=[x_1,...,x_{n-1},\alpha^k(y)]\ \ \forall y\in \mathcal{N}.
\end{equation}
Then

\begin{lem}
The map $\text{ad}^k_X$ is an $\alpha^{k+1}$-derivation, that we call inner $\alpha^{k+1}$-derivation.
\end{lem}
We denote by $Inn_{\alpha^k}(\mathcal{N})$ the space generate by all the inner $\alpha^{k+1}$-derivations.
For any $D\in Der_{\alpha^k}(\mathcal{N})$ and $D'\in Der_{\alpha^k}(\mathcal{N})$ we define their commutator $[D,D']$ as usual:
\begin{equation}\label{DerivationsCommutator}[D,D']=D\circ D'-D'\circ D.\end{equation}
Set $Der(\mathcal{N})=\dl\bigoplus_{k\geq -1}Der_{\alpha^k}(\mathcal{N})$ and $Inn(\mathcal{N})=\dl\bigoplus_{k\geq -1}Inn_{\alpha^k}(\mathcal{N})$.
\begin{definition} An endomorphism $D$ of a multiplicative n-ary Hom-Nambu algebra $(\mathcal N, [~,\cdots,~], \alpha)$  is called a generalized $\alpha^k$-derivation   if there
exist linear mappings $D',D'', \cdots ,D^{(n-1)},D^{(n)} \in End(\mathcal{N}) $ such that
\begin{equation}\label{genereliseDerivetin}
D^{(n)}([x_1, \cdots , x_n])=\sum_{i=1}^n[\alpha^k(x_1), \cdots ,D^{(i-1)}(x_i), \cdots, \alpha^k(x_n)],
\end{equation}
for all $x_1,\cdots , x_n\in\mathcal{N}$. An $(n + 1)$-tuple $(D,D',D'', \cdots ,D^{(n-1)},D^{(n)})$ is called an $(n + 1)$-ary
$\alpha^k$-derivation.
\end{definition}
The set of generalized $\alpha^k$-derivation is denoted by $GDer_{\alpha^k}(\mathcal{N})$. Set  $GDer(\mathcal{N})=\dl\bigoplus_{k\geq -1}GDer_{\alpha^k}(\mathcal{N})$.
\begin{definition}
  Let $(\mathcal N, [~,\cdots,~], \alpha)$ be a multiplicative n-ary Hom-Nambu algebra and $End(\mathcal N)$ be the
endomorphism algebra of $\mathcal N$. An endomorphism $D\in End(\mathcal N)$ is said to be an $\alpha^k-$quasiderivation, if
there exists  an endomorphism $D'\in End(\mathcal N)$ such that
$$\displaystyle\sum_{i=1}^n[\alpha^k(x_1),\dots,D(x_i),\dots, \alpha^k(x_n)]=D'([x_1,\dots, x_n]),$$
for all $x_1,\dots,x_n\in \mathcal N$.
We  call $D'$ the endomorphism associate to the  $\alpha^k-$quasiderivation $D$.
\end{definition}

The set of $\alpha^k$-quasiderivations will be denoted by  $QDer_{\alpha^k}(\mathcal N)$. Set $QDer(\mathcal N)=\dl\bigoplus_{k\geq -1}QDer_{\alpha^k}(\mathcal N)$.

\begin{definition}
  Let $(\mathcal N, [~,\cdots,~], \alpha)$ be a multiplicative n-ary Hom-Nambu algebra and $End(\mathcal N)$ be the
endomorphism algebra of $\mathcal N$. Then the following subalgebra of $End(\mathcal N)$
$$
 Cent(\mathcal N) = \{\theta\in End(N) : \theta([x_1,\dots, x_n])= [\theta(x_1),\dots, x_n], ~~\forall x_i\in\mathcal N\}$$
is said to be the centroid of the n-ary Hom-Nambu algebra.
The definition is the same for classical case of n-ary Nambu algebra. We may also consider the same
definition for any n-ary Hom-Nambu algebra.
\end{definition}
Now, let $(\mathcal N, [~,\cdots,~], \alpha)$ be a multiplicative n-ary Hom-Nambu algebra.
\begin{definition} An $\alpha^k$-centroid of a multiplicative n-ary Hom-Nambu algebra $(\mathcal N, [~,\cdots,~], \alpha)$ is a subalgebra
of $End(\mathcal N)$ denoted $Cent_{\alpha^k}(\mathcal N)$, given by
$$
 Cent_{\alpha^k} (\mathcal N)=\{\theta\in End(\mathcal N): \theta[x_1,\cdots, x_n]=[\theta(x_1), \alpha^k(x_2),\dots, \alpha^k(x_n)], \forall x_i\in\mathcal N\}.$$
We recover the definition of the centroid when $k=0$.
\end{definition}
If $\mathcal N$ is a multiplicative n-Hom-Lie algebra, then it is a simple fact that
$$\theta[x_1,\dots, x_n]=[\alpha^k(x_1),\dots, \theta(x_p),\dots, \alpha^k(x_n)], \forall p \in\{1,\dots, n\}.$$
\section{$n$-Hom-Lie algebras induced by Hom-Lie algebras }
\hspace{0.5cm}In \cite{Arnlind&kitoni&makhlouf&Silvi} and \cite{Arnlind&kitoni&makhlouf}, the authors introduced a construction of a $3$-Hom-Lie algebra from a Hom-Lie
algebra, and more generally of an $(n+1)$-Hom-Lie algebra from an $n$-Hom-Lie algebra. It is
called $(n + 1)$-Hom-Lie algebra induced by $n$-Hom-Lie algebra.
In this context,  Abramov give a new approach of this construction (see  \cite{Abramov2018}). Now we generalise this approach in the Hom case.

\

\hspace{0.5cm}Let  $ (\mathfrak{g},[~,~],\alpha)$ a multiplicative Hom-Lie algebra and $ \mathfrak{g}^*$ be its dual space. Fix an
element of the dual space $\varphi\in \mathfrak{g}^*$. Define the triple product
as follows
\begin{equation}\label{TripleProduct}
[x,y,z]=\varphi(x)[y,z]+\varphi(y)[z,x]+\varphi(z)[x,y],\forall\ x,\ y,\ z\in \mathfrak{g}. \end{equation}
Obviously this triple product is   skew-symmetric. Straightforward computation of the left hand side and the
right hand side of the Filippov-Jacobi identity \eqref{NambuIdentity} if $\varphi\circ\alpha=\varphi$
and
\begin{equation}\label{ConditionTrace}
 \varphi(x)\varphi([y,z])+\varphi(y)\varphi([z,x])+\varphi(z)\varphi([x,y])=0.
\end{equation}
Now we consider $\varphi$ as a $\mathbb{K}$-valued cochain of degree one of the Chevalley-Eilenberg
complex of a Lie algebra $\mathfrak g$. Making use of the coboundary operator $\delta:\wedge^{k}\mathfrak{g}^*\rightarrow\wedge^{k}\mathfrak{g}^*$ defined by
 \begin{equation}\label{OperatorCobord}
    \delta f(u_1,\cdots,u_{k+1})=\sum_{i<j}(-1)^{i+j+1}f([u_i,u_j]_{\mathfrak g},\alpha(u_1)\cdots,\widehat{u_i},\cdots,\widehat{u_j},\cdots,\alpha(u_{k+1})),
 \end{equation}
  for $f\in\wedge^{k}\mathfrak{g}^*$  and for all $ u_1,\cdots,u_{k+1}\in \mathfrak{g}$.
 We obtain that $\delta\varphi(x, y) = \varphi([x, y])$. \\ Finally
we can define  the wedge product of two cochains $\varphi$ and $\delta\varphi$, which is the cochain of
degree three by
$$\varphi\wedge\delta\varphi(x, y, z) =\varphi(x)\varphi([y,z])+\varphi(y)\varphi([z,x])+\varphi(z)\varphi([x,y]).$$
Hence \eqref{ConditionTrace} is equivalent to $\varphi\wedge\delta\varphi=0$.
Thus if an $1$-cochain $\varphi$ satisfies the equation \eqref{ConditionTrace} then the triple product \eqref{TripleProduct} is the
ternary Lie bracket and  we will call this multiplicative 3-Hom-Lie bracket
the quantum Nambu bracket induced by a $1$-cochain.
\begin{definition}
Let $\phi\in \wedge^{n-2}\mathfrak{g}^*$, we define the $n$-ary product as follows
\begin{equation}\label{nProduct}
[x_1,\cdots,x_n]_\phi=\sum_{i<j}^{n}(-1)^{i+j+1}\phi(x_1,\cdots,\hat{x_i},\cdots,\hat{x_j},\cdots,x_n)[x_i,x_j],
\end{equation}
for all $x_1,\cdots,x_n\in\mathfrak{g}$.
\end{definition}
Then, we obtain.
\begin{proposition}
  The $n$-ary product $[~,\cdots,~ ]_\phi$ is skew-symetric.
\end{proposition}
\begin{proof}
  Let $x_1,\cdots,x_n\in \mathfrak g$ and fixed two integer  $i<j$, we have:
  \begin{align*}
  [x_1,\cdots,x_i,\cdots,x_j,\cdots,x_n]_{\phi}=&\displaystyle\sum_{k<l:k,l\neq i,j}(-1)^{k+l+1}\phi(x_1,\cdots,x_i\cdots,\widehat{x}_k,\cdots,x_j,\cdots,\widehat{x}_l,\cdots,x_n)[x_l,x_k]\\
  &+\displaystyle\sum_{i<l\neq j} (-1)^{i+l+1}\phi(x_1,\cdots,\widehat{x}_i\cdots,x_j,\cdots,\widehat{x}_l,\cdots,x_n)[x_i,x_l]\\
  &+\displaystyle\sum_{l< i} (-1)^{i+l+1}\phi(x_1,\cdots,\widehat{x}_l,\cdots,\widehat{x}_i,\cdots,x_j,\cdots,x_n)[x_l,x_i]\\ &+\displaystyle\sum_{j<l} (-1)^{j+l+1}\phi(x_1,\cdots,x_i\cdots,\widehat{x}_j,\cdots,\widehat{x}_l,\cdots,x_n)[x_j,x_l]\\
  &+\displaystyle\sum_{l<j, i\neq l} (-1)^{j+l+1}\phi(x_1,\cdots,x_i,\cdots,\widehat{x}_l,\cdots,\widehat{x}_j,\cdots,x_n)[x_l,x_j]\\
  &+ (-1)^{i+j+1}\phi(x_1,\cdots,\widehat{x}_i,\cdots,\widehat{x}_j,\cdots,x_n)[x_i,x_j]\\
  &=-[x_1,\cdots,x_j,\cdots,x_i,\cdots,x_n]_{\phi}.
  \end{align*}
\end{proof}
Given $X=(x_1,\cdots,x_{n-3})\in \wedge^{n-3}\mathfrak g$, $Y=(y_1,\cdots,y_{n})\in \wedge^{n}\mathfrak g$ and $z\in \mathfrak g$. We define the linear map $\phi_X$ by:
 $$\phi_X(z)=\phi(X,z),$$ and
\begin{align*}
\phi\wedge\delta\phi_X(Y)&=\sum_{i<j}^{n}(-1)^{i+j}\phi(y_1,\cdots \hat{y_i}\cdots\hat{y_j}\cdots,y_{n-1})\delta\phi_X(y_i,y_j)\\
&=\sum_{i<j}^{n}(-1)^{i+j}\phi(y_1,\cdots \hat{y_i}\cdots\hat{y_j}\cdots,y_{n})\phi_X([y_i,y_j]).\end{align*}
\begin{thm}
  Let $(\mathfrak g,[~,~],\alpha)$ be a  multiplicative Hom-Lie algebra, $\mathfrak g^*$ be its dual and $\phi$ be
a cochain of degree $n-2$, i.e. $\phi\in\wedge^{n-2}\mathfrak g^*$. The vector space   $\mathfrak g$
equipped with the n-ary product \eqref{nProduct} and the linear map $\alpha$ is a multiplicative n-Hom-Lie algebra if and only if
\begin{align}\label{NHomLieProduct}
 & \phi\wedge\delta\phi_X=0,\ \forall X\in \wedge^{n-3}\mathfrak{g},\\
  & \phi\circ(\alpha\otimes Id\otimes\cdots\otimes Id)=\phi.
\end{align}
\end{thm}
\begin{proof}
Firstly, let $(x_1,\cdots,x_n)\in  \wedge^{ n}\mathfrak g$. We have:
\begin{align*}
[\alpha(x_1),\cdots,\alpha(x_n)]_{\phi}&=\sum_{i<j}^{n}(-1)^{i+j+1}\phi(\alpha(x_1),
\cdots,\hat{\alpha(x_i)},\cdots,\hat{\alpha(x_j)},\cdots,\alpha(x_n))[\alpha(x_i),\alpha(x_j)]\\
&=\sum_{i<j}^{n}(-1)^{i+j+1}\phi(x_1,
\cdots,\hat{x_i},\cdots,\hat{x_j},\cdots,x_n)\alpha([x_i,x_j])\\&=\alpha([x_1,\cdots,x_n]_{\phi}).
\end{align*}
  Secondly, for $(x_1,\cdots,x_{n-1})\in \wedge^{ n-1}\mathfrak g$ and $(y_1,\cdots,y_n)\in \wedge^{ n}\mathfrak g$, we have
\begin{align*}
   &[\alpha(x_1),\cdots,\alpha(x_{n-1}),[y_1,\cdots,y_n]_{\phi}]_{\phi}\\&=\displaystyle\sum_{i<j}(-1)^{i+j+1} \phi(y_1,\cdots,\widehat{y}_i,\cdots,\widehat{y}_j,\cdots,y_n)[\alpha(x_1),\cdots,\alpha(x_{n-1}),[y_i,y_j]]_{\phi}\\
   &=\displaystyle\sum_{i<j}\displaystyle\sum_{k<l\leq n-1}(-1)^{i+j+k+l}
   \phi(\alpha(x_1),\cdots,\widehat{\alpha(x_k)},\cdots,\widehat{\alpha(x_l)},\cdots,[y_i,y_j])
   \phi(y_1,\cdots,\widehat{y}_i,\cdots,\widehat{y}_j,\cdots,y_n)[\alpha(x_k),\alpha(x_l)]\\
   &+\displaystyle\sum_{i<j}\displaystyle\sum_{k<n}(-1)^{i+j+k}
   \phi(\alpha(x_1),\cdots,\widehat{\alpha(x_k)},\cdots,\alpha(x_{(n-1)}),\cdots,\widehat{[y_i,y_j]})
   \phi(y_1,\cdots,\widehat{y}_i,\cdots,\widehat{y}_j,\cdots,y_n)[\alpha(x_k),[y_i,y_j]].
 \end{align*}
 The terms $[\alpha(x_k),[y_i,y_j]]$ are simplified by identity of Jacobi in the second half of the Filippov identity. Now, we group together the other terms according to their coefficient $[\alpha(x_i),\alpha(x_j)]$. For example, if we fixed $(k, l)$ and, if we collect all the terms containing the commutator $[\alpha(x_k),\alpha(x_l)]$,  then we get the expression
 $$\Big(\displaystyle\sum_{i<j}(-1)^{i+j+k+l}
   \phi(\alpha(x_1),\cdots,\widehat{\alpha(x_k)},\cdots,\widehat{\alpha(x_l)},\cdots,[y_i,y_j])
   \phi(y_1,\cdots,\widehat{y}_i,\cdots,\widehat{y}_j,\cdots,y_n)\Big)[\alpha(x_k),\alpha(x_l)].$$

   Hence the $n$-ary product \eqref{nProduct} will satisfy the $n$-ary Filippov-Jacobi identity if for
any elements $X=(x_1,\cdots,x_{n-3})\in\wedge^{n-3}\mathfrak g$ and $Y=(y_1,\cdots,y_n)\in\wedge^n \mathfrak g$  we require
$$\Big(\displaystyle\sum_{i<j}^n(-1)^{i+j}
   \phi(\alpha(x_1),\cdots,\alpha(x_{n-3}),[y_i,y_j])
   \phi(y_1,\cdots,\widehat{y}_i,\cdots,\widehat{y}_j,\cdots,y_n)\Big)=0.$$
\end{proof}
\begin{definition}
  Let $\phi:\mathfrak g\otimes\cdots\otimes \mathfrak  g\rightarrow \mathbb{K}$ a skew-symmetric multilinear form of   the multiplicative  Hom-Lie algebras $(\mathfrak g,[~,~],\alpha)$, then $\phi$ is called trace if:
   $$\phi\circ(Id\otimes\cdots\otimes Id\otimes[~,~])=0~~~~~\text{and}~~ \phi\circ(\alpha\otimes Id\otimes\cdots\otimes Id)=\phi.
   ~$$
\end{definition}
\begin{cor}
  Let $\phi:\mathfrak g^{\otimes n-2}\rightarrow \mathbb{K}$ be a trace of  Hom-Lie algebra $(\mathfrak g,[~,~],\alpha)$, then $\mathfrak g_\phi=(\mathfrak g,[.,\cdots,.]_\phi,\alpha)$ is a $n$-Hom-Lie algebra.
\end{cor}
%
\begin{proposition}
Let $(\mathfrak g,[~,~],\alpha)$ be a Hom-Lie algebra and $D \in Der(\mathfrak{g})$ be an $\alpha^k$-derivation such that
$$\sum_{i=1}^{n-2}\phi(x_1,\cdots D(x_i),\cdots,x_{n-2})=0.$$
Then $D$ is  an $\alpha^k$-derivation of the $n$-Hom-Lie algebra $(\mathfrak g,[~,\cdots,~]_\phi,\alpha)$.
\end{proposition}
\begin{proof}
Let $X=(x_1,\cdots,x_n)\in\wedge^n \mathfrak g$, on the one hand we get

\begin{align*}
D([x_1,\cdots,x_n]_\phi)&= D\Big(\displaystyle\sum_{i<j}(-1)^{i+j+1}
\phi(\alpha(x_1),\cdots,\alpha(\widehat{x}_i),\cdots,\alpha(\widehat{x}_j),\cdots,\alpha(x_n))[\alpha(x_i),\alpha(x_j)]\Big)\\
&=\displaystyle\sum_{i<j}(-1)^{i+j+1}
\phi(\alpha(x_1),\cdots,\alpha(\widehat{x}_i),\cdots,\alpha(\widehat{x}_j),\cdots,\alpha(x_n))D([\alpha(x_i),\alpha(x_j)])\\
&=\displaystyle\sum_{i<j}(-1)^{i+j+1}
\phi(x_1,\cdots,\widehat{x}_i,\cdots,\widehat{x}_j,\cdots,x_n)[\alpha(D(x_i)),\alpha^{k+1}(x_j)]\\
&+\displaystyle\sum_{i<j}(-1)^{i+j+1}
\phi(x_1,\cdots,\widehat{x}_i,\cdots,\widehat{x}_j,\cdots,x_n)[\alpha^{k+1}(x_i),\alpha(D(x_j))],
\end{align*}
 and on the other hand, we have
 \begin{align*}
&\displaystyle\sum_{l=1}^n[\alpha^k(x_1),\cdots,\alpha^k(x_{l-1}),D(x_l),\cdots,\alpha^k(x_{l+1}),\cdots,\alpha^k(x_n)]_\phi \\&=
\displaystyle\sum_{l=1}^n\displaystyle\sum_{i<j\;;\;i,j\neq l}(-1)^{i+j+1}
\phi(\alpha^k(x_1),\cdots,\widehat{\alpha^k(x_i)},\cdots,D(x_l),\cdots,\widehat{\alpha^k(x_j)},\cdots,\alpha^k(x_n))[\alpha^k(x_i),\alpha^k(x_j)]\\
&+\displaystyle\sum_{l=1}^n\displaystyle\sum_{i<l}(-1)^{i+l+1}
\phi(\alpha^k(x_1),\cdots,\widehat{\alpha^k(x_i)},\cdots,\widehat{D(x_l)},\cdots,\alpha^k(x_n))[\alpha^k(x_i),D(x_l)]\\
&+\displaystyle\sum_{l=1}^n\displaystyle\sum_{l=i<j}(-1)^{j+l+1}
\phi(\alpha^k(x_1),\cdots,\widehat{D(x_l)},\cdots,\alpha^k(x_j),\cdots,\alpha^k(x_n))[D(x_l),\alpha^k(x_j)].
\end{align*}
If $D$ is an $\alpha^k$-derivation then $D([x_1,\cdots,x_n]_\phi)=\displaystyle\sum_{l=1}^n[\alpha^k(x_1),\cdots,\alpha^k(x_{l-1}),D(x_l),\cdots,\alpha^k(x_{l+1}),\cdots,\alpha^k(x_n)]_\phi$, which gives $$\displaystyle\sum_{i<j\;;\;i,j\neq l}(-1)^{i+j+1}\Big(\displaystyle\sum_{l=1}^n\displaystyle
\phi(\alpha^k(x_1),\cdots,\widehat{\alpha^k(x_i)},\cdots,D(x_l),\cdots,\widehat{\alpha^k(x_j)},\cdots,\alpha^k(x_n))\Big)[\alpha^k(x_i),\alpha^k(x_j)]=0.$$ Finally if we fixed $(i,j)$ we have
$$\displaystyle\sum_{l=1}^{n-2}\displaystyle
\phi(\alpha^k(x_1),\cdots,D(x_l),\cdots,\alpha^k(x_{n-2}))=0.$$

\end{proof}
\begin{proposition}

Let $(\mathfrak g,[~,~],\alpha)$ be a Hom-Lie algebra and $D\in QDer(\mathfrak g)$ be an $\alpha^k$-quasiderivation and $D':\mathfrak{g}\rightarrow\mathfrak{g}$ the endomorphism associate to $D$ such that
$$\sum_{i=1}^{n-2}\phi(x_1,\cdots D(x_i),\cdots,x_{n-2})=0.$$
Then $D$ is  an $\alpha^k$-quasiderivation of the $n$-Hom-Lie algebra $(\mathfrak g,[~,\cdots,~]_\phi,\alpha)$ with the same endomorphism associate $D'$.
\end{proposition}

\begin{proposition}

Let $(\mathfrak g,[~,~],\alpha)$ be a Hom-Lie algebra and $\theta:\mathfrak g\rightarrow\mathfrak g$ be an $\alpha^k$-centroid such that
$$\phi(\theta(x_1),\cdots x_i,\cdots,x_{n-2})[\alpha^k(x),y]=\phi(x_1,\cdots x_i,\cdots,x_{n-2})[\theta(x),y].$$
Then $D$ is  an $\alpha^k$-centroid on the $n$-Hom-Lie algebra $(\mathfrak g,[~,\cdots,~]_\phi,\alpha)$.
\end{proposition}
\begin{proof}
Let $x_1,\cdots,x_n\in\mathfrak{g}$, we have
\begin{align*}
\theta([x_1,\cdots,x_n]_\phi)&=\sum_{i<j}^{n}(-1)^{i+j+1}\phi(x_1,\cdots,\hat{x_i},\cdots,\hat{x_j},\cdots,x_n)\theta([x_i,x_j])\\
&=\sum_{i<j}^{n}(-1)^{i+j+1}\phi(x_1,\cdots,\hat{x_i},\cdots,\hat{x_j},\cdots,x_n)[\theta(x_i),\alpha^k(x_j)].
\end{align*}
On the other hand, we have
\begin{align*}
[\theta (x_1),\alpha^k(x_2),\cdots,\alpha^k(x_n)]_\phi&=\sum_{i<j}^{n}(-1)^{i+j+1}\phi(\theta(x_1),\alpha^k(x_2),\cdots,\hat{x_i},\cdots,\hat{x_j},\cdots,\alpha^k(x_n))[\alpha^k(x_i),\alpha^k(x_j)]\\
&
=\sum_{i<j}^{n}(-1)^{i+j+1}\phi(x_1,\cdots,\hat{x_i},\cdots,\hat{x_j},\cdots,x_n)[\theta(x_i),\alpha^k(x_j)]\\&=\theta([x_1,\cdots,x_n]_\phi).
\end{align*}
\end{proof}

\section{Hom-Lie $n$-uplet system}
\subsection{Hom-Lie triple system}

\

In this  section, we start by recalling the definitions of Lie triple systems and Hom-Lie triple systems.
\begin{definition}
 (\cite{Lister})

A vector space $T$ together with a trilinear map $(x, y, z)\rightarrow[x,y,z]$ is called a
Lie triple system (LTS) if
\begin{enumerate}
  \item  $[x,x,z]=0,$
  \item $[x,y,z]+[y,z,x]+[z,x,y]=0$,
  \item $[u,v,[x,y,z]]=[[u,v,x],y,z]+[x,[u,v,y],z]+[x,y,[u,v,z]],$
\end{enumerate}
for all $x,y,z,u,v\in T$.
\end{definition}
\begin{definition}(\cite{Yau2012})
 A Hom-Lie triple system (Hom-LTS for short) is denoted by
$(T,[\cdot,\cdot,\cdot], \alpha)$, which consists of an $\mathbb{K}$-vector space $T$, a trilinear product $[\cdot,\cdot,\cdot]: T\times T\times T\rightarrow T$, and a linear map $\alpha:T\rightarrow T$, called the twisted map, such that $\alpha$ preserves the product and for all
$x,y,z,u,v\in T$,
\begin{enumerate}
  \item  $[x,x,z]=0,$
  \item $[x,y,z]+[y,z,x]+[z,x,y]=0$,
  \item $[\alpha(u),\alpha(v),[x,y,z]]=[[u,v,x],\alpha(y),\alpha(z)]+[\alpha(x),[u,v,y],\alpha(z)]+[\alpha(x),\alpha(y),[u,v,z]]$.
\end{enumerate}
\end{definition}

\begin{remark}
When  the twisted map $\alpha$ is equal to the identity map, a Hom-LTS is a LTS. So
LTS are special examples of Hom-LTS.
\end{remark}
\begin{definition} A Hom-Lie triple system
$(T,[\cdot,\cdot,\cdot], \alpha)$ is called multiplicative  if $\alpha([x,y,z])=[\alpha(x),\alpha(y),\alpha(z)]$, for all $x,y,z\in T$.
\end{definition}

\begin{thm}$($\cite{Yau2012}$)$\label{theoremDY}

Let $(\mathfrak{g},[\cdot,\cdot], \alpha)$ be a multiplicative  Hom-Lie algebra. Then
$$
\mathfrak{g}_T=(\mathfrak{g},[\cdot,\cdot,\cdot]=[\cdot,\cdot]\circ([\cdot,\cdot]\otimes \alpha), \alpha^2),$$
is a multiplicative Hom-Lie triple system.
\end{thm}\subsection{Hom-Lie $n$-uplet system}

\

In this  section we introduce  the   definitions of Lie $n$-uplet systems and multiplicative  Hom-Lie $n$-uplet systems. We give the analogue of Theorem \ref{theoremDY} in the Hom-Lie $n$-uplet systems case.
\begin{definition}

A vector space $\mathcal{G}$ together with a $n$-linear map $(x_1,\cdots, x_n)\rightarrow[x_1,\cdots, x_n]$   is called a
Lie $n$-uplet system  if
\begin{enumerate}
  \item  $[x,x,y_1,\cdots,y_{n-2}]=0,$ for all $x,y_1,\cdots,y_{n-2}\in \mathcal{G}$.
  \item $ \big[x_1,....,x_{n-1},[y_1,....,y_{n}]\big]=
 \dl\sum_{i=1}^{n}\big[y_1,....,y_{i-1},[x_1,....,x_{n-1},y_i]
  ,y_{i+1},...,y_n\big],$
\end{enumerate}
for all $x_1,\cdots,x_{n-1},y_1,\cdots,y_{n}\in \mathcal{G}$.
\end{definition}\begin{definition}

A vector space $\mathcal{G}$ together with a $n$-linear map $(x_1,\cdots, x_n)\rightarrow[x_1,\cdots, x_n]$ and a family
$\widetilde{\alpha}=(\alpha_i)_{1\leq i\leq n-1}$ of  linear maps $ \alpha_i:\ \ \mathcal{G}\longrightarrow \mathcal{G}$, is called a
Hom-Lie $n$-uplet system  if
\begin{enumerate}
  \item  $[x,x,y_1,\cdots,y_{n-2}]=0,$ for all $x,y_1,\cdots,y_{n-2}\in \mathcal{G}$.
  \item $ \big[\alpha_1(x_1),\dots,\alpha_{n-1}(x_{n-1}),[y_1,\dots,y_{n}]\big]=$\\ $  \dl\sum_{i=1}^{n}\big[\alpha_1(y_1),\dots,\alpha_{i-1}(y_{i-1}),[x_1,\dots,x_{n-1},y_i]
  ,\alpha_i(y_{i+1}),\dots,\alpha_{n-1}(y_n)\big],$\\
  for all $x_1,\cdots,x_{n-1},y_1,\cdots,y_{n}\in \mathcal{G}$.
\end{enumerate}
\end{definition}
\begin{definition} A Hom-Lie $n$-uplet system $(\mathcal{G},[~,\cdots, ~],\widetilde{\alpha})$ is called a multiplicative Hom-Lie $n$-uplet system if $\alpha_1=\dots=\alpha_{n-1}=\alpha$ and $\alpha([x_1,\cdots, x_n])=[\alpha(x_1),\cdots, \alpha(x_n)]$ for all $x_1,\cdots, x_n\in \mathcal G$.
\end{definition}

\begin{remark}
When  the twisted maps $\alpha_i$ is equal to the identity map, a Hom-Lie $n$-uplet systems is a Lie $n$-uplet systems. So
Lie $n$-uplet systems are special examples of Hom-Lie $n$-uplet systems.
\end{remark}
 Let $(\mathfrak{g},[~,~],\alpha)$ be a Hom-Lie algebra. We define the following $n$ linear map:
 \begin{eqnarray}
 [~,\cdots,~]_n:\mathfrak g^{\otimes n}&\longrightarrow& \mathfrak g\nonumber\\
(x_1,\cdots,x_n)&\longmapsto& \Big[x_1,\cdots,x_n\Big]_n=\Big[\big[[\dots[x_1,x_2],\alpha(x_3)],\alpha^2(x_4)\big]\cdots\alpha^{n-3}(x_{n-1})],\alpha^{n-2}(x_{n})\Big]\label{crochet_n}
\end{eqnarray}
For n=2, $[x_1,x_2]_2=[x_1,x_2]$ and for $n\geq 3$ we have $[x_1,\cdots,x_n]_n=[[x_1,\cdots,x_{n-1}]_{n-1},x_n]$.\\

The following result  gives a way to construct Hom-Lie $n$-uplet system starting from
a classical Lie $n$-uplet system and algebra endomorphisms.
\begin{proposition}Let $(\mathcal{G},[~,\cdots, ~])$ a Lie $n$-uplet system and $\alpha :\mathcal{G}\rightarrow\mathcal{G}$ a linear map such that $\alpha([x_1,\cdots,x_n])=[\alpha(x_1),\cdots,\alpha(x_n)]$.
Then $(\mathcal{G},[~,\cdots, ~]_\alpha,\alpha)$ be a Hom-Lie $n$-uplet system, where $[x_1,\cdots,x_n]_\alpha=[\alpha(x_1),\cdots,\alpha(x_n)]$, for all $x_1,\cdots,x_n\in \mathcal{G}$. \end{proposition}

\begin{thm}\label{n-upletfrom Lie}
Let $(\mathfrak{g},[\ ,\ ], \alpha)$ be a multiplicative  Hom-Lie algebra. Then
$$
\mathfrak{g}_n=(\mathfrak{g},[\ ,\cdots,\ ]_n, \alpha^{n-1})$$
is a multiplicative Hom-Lie $n$-uplet system.

\end{thm}
When $n=3$ we obtain the multiplicative Hom-Lie triple system constructed in Theorem \ref{theoremDY}.
To prove this theorem we need to the following lemma.
\begin{lem}\label{adj2}
Let $(\mathfrak{g},[\ ,\ ], \alpha)$ be a multiplicative Hom-Lie algebra, and $\text{ad}^2$ the adjoint map defined by
$$\text{ad}_x^2(y)=\text{ad}_x(y)=[x,y]$$ Then, we have
$$\text{ad}_{\alpha^{n-1}(x)}^2[y_1,\cdots,y_n]_n=\displaystyle\sum_{k=1}^n
[\alpha(y_1),\cdots,\alpha(y_{k-1}),\text{ad}_x^2(y_k),\alpha(y_{k+1}),\cdots,\alpha(y_n)]_n,$$
where $x\in\mathfrak{g}, y\in \mathfrak g$ and $(y_1,\cdots,y_n)\in\mathfrak{g}^n$.

\end{lem}
\begin{proof}
  For $n=2$, using the Hom-Jacobi  identity we have
  \begin{align*}
\text{ad}_{\alpha(x)}^2[y,z]&=[\alpha(x),[y,z]]=[[x,y],\alpha(z)]+[\alpha(y),[x,z]]\\&=[\text{ad}_x^2(y),\alpha(z)]+[\alpha(y),\text{ad}_x^2(z)].
  \end{align*}
 Assume that the property is true up to order $n$, that is $$\text{ad}_{\alpha^{n-1}(X)}^2[y_1,\cdots,y_n]_n=\displaystyle\sum_{k=1}^n
[\alpha(y_1),\cdots,\alpha(y_{k-1}),\text{ad}_X^2(y_k),\alpha(y_{k+1}),\cdots,\alpha(y_n)]_n.$$
Let $x\in\mathfrak{g}$ and $(y_1,\cdots,y_{n+1})\in\mathfrak{g}^{n+1}$, we have
\begin{align*}
\text{ad}^2_{\alpha^n(x)}[y_1,\cdots,y_{n+1}]&=\text{ad}^2_{\alpha^n(x)}[[y_1,\cdots,y_n]_n,\alpha^{n-1}(y_{n+1})]_2\\&=
\Big[\text{ad}^2_{\alpha^{n-1}(x)}[y_1,\cdots,y_n]_n,\alpha^n(y_{n+1})\Big]_2+
\Big[[\alpha(y_1),\cdots,\alpha(y_n)]_n,\text{ad}^2_{\alpha^{n-1}(x)}(\alpha^{n-1}(y_{n+1}))\Big]_2\\&=
\displaystyle\sum_{k=1}^n\Big[[\alpha(y_1),\cdots,\alpha(y_{k-1}),\text{ad}^2_x(y_k),\alpha(y_{k+1}),\cdots,\alpha(y_n)]_n,\alpha^n(y_{n+1})\Big]\\&+
\Big[[\alpha(y_1),\cdots,\alpha(y_n)]_n,\alpha^{n-1}(\text{ad}^2_x(y_{n+1}))\Big]_2\\&=
\displaystyle\sum_{k=1}^n\Big[\alpha(y_1),\cdots,\alpha(y_{k-1}),\text{ad}^2_x(y_k),\alpha(y_{k+1}),\cdots,\alpha(y_n),\alpha(y_{n+1})\Big]_{n+1}\\&+
\Big[\alpha(y_1),\cdots,\alpha(y_n),\text{ad}^2_x(y_{n+1})\Big]_{n+1}\\&=
\displaystyle\sum_{k=1}^{n+1}\Big[\alpha(y_1),\cdots,\alpha(y_{k-1}),\text{ad}^2_x(y_k),\alpha(y_{k+1}),\cdots,\alpha(y_{n+1})\Big]_{n+1}.
\end{align*}
The lemma is proved.
\end{proof}

\begin{proof}(\textbf{Proof of Theorem \ref{n-upletfrom Lie}})
Let $X=(x_1,\cdots,x_{n-1})\in\mathfrak{g}^{n-1}$ and $Y=(y_1,\cdots,y_n)\in\mathfrak{g}^n$.\\

(i) Its easy to see that $[x_1,x_1,x_2,\cdots,x_{n-1}]_n=[[\cdots[[x_1,x_1]_2,\alpha(x_2)]_2,\alpha^2(x_3)]_2,\cdots]_2,\alpha^{n-2}(x_{n-1})]_2=0$\\

(ii) Using lemma \eqref{adj2}, we have
\begin{align*}
\Big[\alpha^{n-1}(x_1),\cdots,\alpha^{n-1}(x_{n-1}),[y_1,\cdots,y_n]_n\Big]_n&=
\Big[[\alpha^{n-1}(x_1),\cdots,\alpha^{n-1}(x_{n-1})]_{n-1},[\alpha^{n-2}(y_1),\cdots,\alpha^{n-2}(y_n)]_n\Big]_2\\&=
\text{ad}^2_{\alpha^{n-1}[x_1,\cdots,x_{n-1}]}([\alpha^{n-2}(y_1),\cdots,\alpha^{n-2}(y_n)]_n)\\&=
\displaystyle\sum_{k=1}^n\Big[\alpha^{n-1}(y_1),\cdots,\text{ad}^2_{[x_1,\cdots,x_{n-1}]}(\alpha^{n-2}(y_k)),\cdots,\alpha^{n-1}(y_n)\Big]_n\\&=
\displaystyle\sum_{k=1}^n\Big[\alpha^{n-1}(y_1),\cdots,[[x_1,\cdots,x_{n-1}],\alpha^{n-2}(y_k)]_2,\cdots,\alpha^{n-1}(y_n)\Big]_n\\&=
\displaystyle\sum_{k=1}^n\Big[\alpha^{n-1}(y_1),\cdots,[x_1,\cdots,x_{n-1},y_k]_n,\cdots,\alpha^{n-1}(y_n)\Big]_n.
\end{align*}

\end{proof}
\begin{example}

Using Exemple \ref{Sl2HomLieExemple} and by the theorem \eqref{n-upletfrom Lie}, for $\lambda\in\mathbb{R}^*$ we have:\\

For $n=3,\;(\mathfrak{sl}_2(\mathbb{R}),[~,~,~]_3,\alpha^2_\lambda)$ is a Hom-Lie triple system. The different brackets are as follows:\\

$[H,X,Y]_3=[[H,X]_{\alpha_{\lambda}},\alpha_{\lambda}(Y)]_{\alpha_{\lambda}}=2H $\;\;\;;\;\;\;$[H,X,H]_3=-4\lambda^4X$\;\;\;;\;\;\;$[H,Y,X]_3=4H$.\\

$[H,Y,H]_3=-\frac{4}{\lambda^4}Y$\;\;\;;\;\;\;$[X,Y,Y]_3=-\frac{2}{\lambda^4}Y$\;\;\;;\;\;\;$[X,Y,X]_3=2\lambda^4X$.\\

The other brackets are zeros.\\

For $n=4,\;(\mathfrak{sl}_2(\mathbb{R}),[~,~,~,~]_4,\alpha^3_\lambda)$ is a $4$-uplet Hom-Lie triple system. The different brackets are defined as follows:\\

$[H,X,H,H]_4=[[H,X,H]_3,\alpha^2(H)]_{\alpha_{\lambda}}=-4\lambda^4[X,H]_{\alpha_{\lambda}}=8\lambda^6X$\;\;\;;\;\;\;$[H,X,H,Y]_4=-4H$\\

$[H,Y,H,H]_4=-\frac{8}{\lambda^6}Y$\;\;\;;\;\;\;$[H,Y,H,X]_4=4H$\;\;\;;\;\;\;$[H,X,Y,X]_4=4\lambda^6X$\;\;\;;\;\;\;$[H,X,Y,Y]_4=-\frac{2}{\lambda^6}Y$\\

$[H,Y,X,X]_4=8\lambda^6X$\;\;\;;\;\;\;$[H,Y,X,Y]_4=-\frac{8}{\lambda^6}Y$\;\;\;;\;\;\;$[X,Y,X,Y]_4=2H$\\

$[X,Y,X,H]_4=-4\lambda^6X$\;\;\;;\;\;\;$[X,Y,Y,X]_4=2H$\;\;\;;\;\;\;$[X,Y,Y,H]_4=-\frac{4}{\lambda^6}Y$.\\

The other brackets are zeros.\\

\end{example}

\begin{proposition}
  Let $(\mathfrak{g},[\ ,\ ], \alpha)$ be a multiplicative Hom-Lie algebra and $D:\mathfrak{g}\rightarrow\mathfrak{g}$ an $\alpha^k$-derivation of $\mathfrak{g}$ for an integer $k$. Then $D$ is an $\alpha^k$-derivation of $\mathfrak{g}_n$.
\end{proposition}
\begin{proof}
By recurrence\\

  For $n=3$.
  Let $x,y,z\in\mathfrak{g}$, we have
  \begin{align*}
  D([x,y,z])&=D([[x,y],\alpha(z)])\\&=[D([x,y]),\alpha^{k+1}(z)]+[[\alpha^k(x),\alpha^k(y)],D(\alpha(z))]\\&=[[D(x),\alpha^k(y)],\alpha^{k+1}(z)]+
  [[\alpha^k(x),D(y)],\alpha^{k+1}(z)]+[[\alpha^k(x),\alpha^k(y)],\alpha(D(z))]\\&=[D(x),\alpha^k(y),\alpha^k(z)]+[\alpha^k(x),D(y),\alpha^k(z)]+
  [\alpha^k(x),\alpha^k(y),D(z)]
  \end{align*}
  Now, suppose that the property is true to order $n-1$, i.e: $$D([x_1,\cdots,x_{n-1}]_{n-1})=\displaystyle\sum_{i=1}^n[\alpha^k(x_1),\cdots,D(x_k),\cdots,\alpha^k(x_{n-1})]_{n-1}.$$
   Let $(x_1,\cdots,x_n)\in\mathfrak{g}^n$, then
  \begin{align*}
D([x_1,\cdots,x_n]_n)&=D(\Big[[x_1,\cdots,x_{n-1}]_{n-1},\alpha^{n-2}(x_n)\Big])\\&=
\Big[D([x_1,\cdots,x_{n-1}]_{n-1}),\alpha^{n+k-2}(x_n)\Big]+\Big[[\alpha^k(x_1),\cdots,\alpha^k(x_{n-1})]_{n-1},D(\alpha^{n-2}(x_n))\Big]\\&=
  \Big[D([x_1,\cdots,x_{n-1}]_{n-1}),\alpha^{n-2}(\alpha^k(x_n))\Big]+\Big[[\alpha^k(x_1),\cdots,\alpha^k(x_{n-1})]_{n-1},\alpha^{n-2}(D(x_n))\Big]\\&=
  \displaystyle\sum_{i=1}^{n-1}\Big[[\alpha^k(x_1),\cdots,D(x_i),\cdots,\alpha^k(x_{n-1})]_{n-1},\alpha^{n-2}(\alpha^k(x_n))\Big]+
  [\alpha^k(x_1),\cdots,\alpha^k(x_{n-1}),D(x_n)]_n\\&=
  \displaystyle\sum_{i=1}^{n-1}[\alpha^k(x_1),\cdots,D(x_i),\cdots,\alpha^k(x_{n-1}),\alpha^k(x_n)]_n+
  [\alpha^k(x_1),\cdots,\alpha^k(x_{n-1}),D(x_n)]_n\\&=
  \displaystyle\sum_{i=1}^n[\alpha^k(x_1),\cdots,D(x_i),\cdots,\alpha^k(x_{n-1}),\alpha^k(x_n)]_n
  \end{align*}

\end{proof}
\begin{proposition}Let $(\mathfrak{g},[\ ,\ ], \alpha)$ be a multiplicative Hom-Lie algebra and $D,D',\cdots,D^{(n-1)}$ an endomorphisms of $\mathfrak{g}$ such that $D^{(i)} $ is an $\alpha^k$-quasiderivation with associated endomorphism $D^{(i+1)} $ for  $0\leq i\leq n-2$. Then the $(n + 1)$-tuple $(D,D,D',D'', \cdots ,D^{(n-1)})$ is an $(n + 1)$-ary
$\alpha^k$-derivation of $\mathfrak{g}_{n}$.
\end{proposition}
\begin{proof}
Let $x_1,\cdots,x_n\in \mathfrak{g}$, then
\begin{align*}
D^{(n-1)}([x_1,\cdots,x_n]_n)&=D^{(n-1)}([[x_1,\cdots,x_{n-1}]_{n-1},x_n])\\
&=[D^{(n-2)}([x_1,\cdots,x_{n-1}]_{n-1}),\alpha^k(x_n)]+[[\alpha^k(x_1),\cdots,\alpha^k(x_{n-1})]_{n-1},D^{(n-2)}(x_n)]\\
&\vdots\\
&=[D(x_1),\alpha^k(x_2),\cdots,\alpha^k(x_n)]_n+[\alpha^k(x_1),D(x_2),\cdots,\alpha^k(x_n)]_n\\
&+[\alpha^k(x_1),\alpha^k(x_2),D'(x_3),\cdots,\alpha^k(x_n)]_n+\cdots+[\alpha^k(x_1),\cdots,\alpha^{k}(x_{n-1}),D^{(n-2)}(x_n)]_n)
\end{align*}
Therefore  $(n + 1)$-tuple $(D,D,D',D'', \cdots ,D^{(n-1)})$ is an $(n + 1)$-ary
$\alpha^k$-derivation of $\mathfrak{g}_{n}$.
\end{proof}

\noindent\textbf{Acknowledgments}

We would like to thank Abdenacer Makhlouf and Viktor Abramov for helpful discussions and for their interest in this work.

\end{document}